\RequirePackage[l2tabu, orthodox]{nag}
\documentclass[a4paper,11pt]{amsart}

\usepackage[utf8]{inputenc}
\usepackage[T1]{fontenc}
\usepackage{tikz-cd} 
\usepackage[ruled,vlined]{algorithm2e}

\usepackage{amsmath,amsfonts,amssymb,amsxtra,setspace,xspace,graphicx,lmodern,psfrag,color,latexsym,bbm,comment,amsthm,enumerate,enumitem,pifont,mathrsfs,caption,subcaption,microtype,dsfont, bm,nicefrac}
\usepackage[foot]{amsaddr}
\usepackage{tkz-euclide}

\usepackage[a4paper]{geometry}

\usepackage[hidelinks,colorlinks=true]{hyperref}\hypersetup{urlcolor=blue, citecolor=blue}

\usepackage{mathtools}

\usepackage[capitalise]{cleveref}
\usepackage{autonum}
\numberwithin{equation}{section}
\makeatletter
 \def\@textbottom{\vskip \z@ \@plus 17pt}
 \let\@texttop\relax
\makeatother
\usepackage[acronym,nonumberlist,nogroupskip]{glossaries}

\usepackage{array,booktabs}

\newtheorem{thm}{Theorem}[section]
\newtheorem{cor}[thm]{Corollary}
\newtheorem{lem}[thm]{Lemma}
\newtheorem{prop}[thm]{Proposition}

\theoremstyle{plain}
\theoremstyle{definition}

\theoremstyle{remark}

\Crefname{lem}{Lemma}{Lemmas}

\def\XXint#1#2#3{{\setbox0=\hbox{$#1{#2#3}{\int}$ }
\vcenter{\hbox{$#2#3$ }}\kern-.6\wd0}}

\renewcommand{\glossarysection}[2][]{} 

\DeclarePairedDelimiter{\abs}{\lvert}{\rvert}
\DeclarePairedDelimiter{\norm}{\lVert}{\rVert}
\DeclarePairedDelimiter{\bra}{(}{)}
\DeclarePairedDelimiter{\pra}{[}{]}
\DeclarePairedDelimiter{\set}{\{}{\}}

\DeclareMathAlphabet{\mathup}{OT1}{\familydefault}{m}{n}
\newcommand{\dx}[1]{\mathop{}\!\mathup{d} #1}

\newcommand{\N}{{\mathbb N}}
\newcommand{\R}{{\mathbb R}}

\newcommand{\eps}{{\varepsilon}}

\newcommand{\al}{\alpha}

\DeclareMathOperator{\supp}{supp}

\renewcommand{\tilde}{\widetilde}

\renewcommand{\eps}{\varepsilon}
\allowdisplaybreaks[4] 

\author{Florian Kunick}
\address{Max-Planck-Institut f\"ur Mathematik in den Naturwissenschaften}
\email{kunick@mis.mpg.de}
\date{\today}

\begin{document}

\title{A first order description of a nonlinear SPDE in the spirit of rough paths}
\maketitle
\begin{abstract}
    We consider a nonlinear stochastic partial differential equation (SPDE) in divergence form where the forcing term is a Gaussian noise, that is white in time and colored in space such that the gradient of the solution is H\"older-continuous, but not differentiable.
    Then, we prove a generalized Taylor expansion of the difference between the solution to the SPDE and the solution to its linearization around a fixed basepoint. The result is reminiscent of the theory of (controlled) rough paths and agrees with the general observation, that, in settings with a rough driver, subtracting the solution to the linearized equation yields a more regular object. 
\end{abstract}

\section{Introduction}
Let $A: \mathbb{R}^d \to \mathbb{R}^d$ be given. We consider the following stochastic partial differential equation (SPDE)
\begin{align} \label{maineq}
    \begin{cases}
	    \partial_t u = \nabla \cdot A(\nabla u) + \xi \\
	    u|_{t \leq 0} = 0
    \end{cases}
\end{align}
on $\R \times \R^d$. The forcing term $\xi$ is a space-time Gaussian noise, which is white in time, and periodic, colored and stationary in space. More precisely, we consider a Gaussian process $\xi$, formally defined via its covariance
\begin{align}
    \mathbb{E}\pra*{\xi(t, x) \xi(t', x')} = \delta(t-t')K(x-x'),
\end{align}
where $K: \R^d \to \R$ is periodic (cf. \cref{sec:she}), and that for convenience is localized in the time interval $\pra*{0,1}$. 
The spatial covariance function $K$ is chosen in such a way that the solution $v$ to the linearized equation, i.e. the stochastic heat equation (SHE)
\begin{align} \label{heat}
    \begin{cases}
        \partial_t v =  \Delta v + \xi \\
    	v|_{t \leq 0} = 0,
    \end{cases}
\end{align}
is not only a (continuous) function, but differentiable, and we have
\begin{align}\label{vhoelder}
	\pra*{\nabla v}_{\al} < \infty \ \mathrm{a.s.}
\end{align}
for some $\alpha \in \bra*{\frac{1}{2}, 1}$\footnote{This is in contrast to the SHE with space-time white noise forcing, where already in dimension 2 the solution is not a function anymore}. For the precise notation we refer to \cref{sec:not}. Since $v$ is a linear functional of the Gaussian field $\xi$ -- and thus is itself Gaussian -- more is true. Indeed, in \cite[Lemma 3]{OW19} Gaussian moments for $v$ were established, meaning that there exists a $C_0 > 0$ such that
\begin{align}\label{vgaussian}
    \mathbb{E}\pra*{\mathrm{exp}\bra*{\frac{1}{C_0} \pra*{\nabla v}^2_{\al}}} < \infty.
\end{align}
Moreover, we assume that the non-linearity $A$ is elliptic (cf. \eqref{ellA}). Under these assumptions, it is expected that we should also have
\begin{align}\label{uhoelder}
    \pra*{\nabla u}_{\al} < \infty \ \mathrm{a.s.}
\end{align}
and in fact \cite[Corollary 1]{OW19} yields the (a priori) estimate
\begin{align}\label{aprioriest}
    \pra*{\nabla u}_{\al} \lesssim \pra*{\nabla v}_{\al}^{\frac{\al}{\al_0}} + \pra*{\nabla v}_{\al}
\end{align}
for an $\al_0 \in \bra*{0, 1}$ as in \cite[Lemma 1]{OW19}\footnote{coming from the DeGiorgi--Nash theorem}. Here, we note that the implicit constant in $\lesssim$ only depends and will depend on the dimension, $\al$ and the ellipticity of $A$ (cf. \cref{sec:not}).

\medskip

In particular, we want to stress that for spatially colored noise, \eqref{maineq} is \emph{not} a singular SPDE.
In fact, the SPDE \eqref{maineq} with space-time white noise is, in general, not subcritical in any space-dimension, in the sense that by zooming in on small scales, the nonlinear terms blow up. For a reference concerning singular (semi-linear) SPDEs and the notion of subcriticality we refer to \cite{H14}. 

\medskip

We should mention, that there are no probabilistic arguments in this paper and if $\xi$ is a distribution of suitable regularity such that \eqref{vhoelder} holds, the same result can be obtained. Nevertheless, we have chosen to consider the case where $\xi$ is a (Gaussian) noise, on the one hand because we are heavily building on the work \cite{OW19}\footnote{where the analysis is also purely deterministic except for \eqref{vgaussian}} and on the other hand because a typical example of a distribution with a negative H\"older-regularity is a realization of some Gaussian field with a short correlation length. Thus, everything in this article can be seen as a pathwise analysis, which is in the spirit of the theory of regularity structures (cf. \cite{H14}), where there is a clear distinction between the deterministic and probabilistic steps. Usually, the probabilistic arguments involve the construction of singular products along with the corresponding renormalization. In our setting no renormalization is necessary and hence it appears to be natural that the arguments are purely deterministic.

\medskip

In \cite[p.70, Theorem 1]{OW19}, it is shown that, under suitable assumptions on $A$ -- that for convenience we recall in \cref{sec:not} -- a (unique) solution to \eqref{maineq} exists such that \eqref{aprioriest} and thus \eqref{uhoelder} hold. Due to the rough setting, the authors introduce a spatial increment operator $\delta_y$ (cf. \eqref{increment}) with which they linearize equation \eqref{maineq}\footnote{as opposed to taking the derivative of the equation due to the low regularity of the noise}. Then it is convenient (and a guiding principle) to subtract the increment $\delta_y v$ (where $v$ solves \eqref{sthe}) in order to get rid of the noise $\xi$.
Hence $\delta_y\bra*{u - v}$ satisfies a variable, but linear coefficient equation and the celebrated DeGiorgi--Nash theorem\footnote{a localized version thereof, to be precise} yields an a priori estimate for its H\"older-norm for some $\al_0 \in \bra*{0, 1}$. Postprocessing this estimate in turn establishes \eqref{uhoelder} for $\al = \al_0$. In the next step, this estimate is upgraded using standard $C^{1 + \al}$-Schauder theory as well as the stochastic estimate \eqref{vgaussian}, thus yielding \eqref{uhoelder} for any $\al \in (0,1)$. In general, \eqref{uhoelder} is sharp; in case of Brownian motion it is well-known that its paths have H\"older--regularity of at most $\frac{1}{2}$ but not better. We want to address the following question: Is there a way to give a finer description of the regularity of $u$? 

\medskip

In this work we intend to give such a regularity statement in the following way. 
It is often the case that the difference of the solution to a non-linear equation with a rough driving signal and the solution to the linearized equation is more regular. A prominent example that has been treated in the recent years is the $\phi^4_2$-model (cf. \cite{DPD03}). Indeed, its solution $\varphi$ and the solution $f$ to the stochastic heat equation in dimension $2$ are distributions but their difference $\varphi - f$ is smoother, in particular a function. This has been exploited in \cite{DPD03} in order to construct a solution which is local in time. In fact, this is also a guiding principle for the theory of (controlled) rough paths (cf. \cite{G04}) and ultimately in the theory of regularity structures (cf. \cite{H14}), where this concept has been vastly generalized. 

\medskip

Our main result is that $\nabla u$ is \emph{modelled} after $\nabla v$ which essentially means that a certain \emph{a priori}-estimate holds (cf. \eqref{modconstant}). In the context of singular SPDEs such a modelling assumption is used in the following way. First, one constructs a singular product  on the level of the linear, but irregular model, i.e. $\nabla v$ in this case. In a second step, one constructs the singular product on the level of the solution $\nabla u$ using the singular product on the level of the model and the assumption that $\nabla u$ is modelled after $\nabla v$.
Due to the nonlinear nature of the problem it is not sufficient to just consider the solution $v$ to the SHE, but for all space-time points $z = (t', x')$ we consider the solution to an anisotropic SHE.

\medskip

Let $a(t', x') := DA(\nabla u(t', x')) \in \R^{d \times d}$. Then we write $v_{a(t', x')}$ for the solution to the anisotropic stochastic heat equation, i.e. $v_{a(t', x')}$ solves
\begin{align}\label{ashe}
    \begin{cases}
        \partial_t v_{a(t', x')} = \nabla \cdot a(t', x') \nabla v_{a(t', x')} + \xi \\
        v_{a(t', x')}|_{t \leq 0} = 0
    \end{cases}
\end{align}
with $\xi$ as in \eqref{heat}. By the $C^{1+ \al}$-Schauder theory developed in \cite{OW19} we get the following uniform\footnote{in the basepoint} estimate.
\begin{lem}\label{Lem1}
    There exists a solution $v_{a(t', x')}$ to \eqref{ashe} where $a(t', x') = DA(\nabla u(t', x'))$ and $A$ satisfies \eqref{ellA}, \eqref{boundedA} as well as \eqref{Lip}, and we have 
    \begin{align}
        \sup_{\bra*{t', x'} \in \R \times \R^d} \pra*{\nabla v_{a(t', x')}}_{\al} \lesssim \pra*{\nabla v}_{\al}^{\frac{\al}{\al_0}} + \pra*{\nabla v}_{\al}.
    \end{align}
\end{lem}

\medskip

Then we can state our main result.
\begin{thm}\label{Thm1}
Let $u$ be a solution to the equation
\begin{align}
    \begin{cases}
        \partial_t u = \nabla \cdot A(\nabla u) + \xi \\
        u|_{t \leq 0} = 0
    \end{cases}
\end{align}
on $\R \times \R^d$ where $A$ satisfies \eqref{ellA}, \eqref{boundedA} and \eqref{Lip}, and where $\xi$ is a Gaussian noise that is white in time, periodic, stationary and colored in space such that the corresponding solution $v$ to the stochastic heat equation \eqref{heat} satisfies $[\nabla v]_{\al} < \infty$ for some $\al \in (\frac{1}{2}, 1)$. Moreover, let $a(t', x') = DA(\nabla u(t', x')) \in \R^{d \times d}$ and let $v_{a(t', x')}$ be a solution to \eqref{ashe}.
Then there exists a family of symmetric matrices $(B(t', x'))_{(t', x')}$ such that for all $x, x' \in \mathbb{R}^d$ and $t, t' \in \R$ it holds a.s. that
\begin{align}\label{est1}
    \abs*{\nabla u (t, x) - \nabla u(t', x') - \bra*{\nabla v_{a(t', x')}(t, x) - \nabla v_{a(t', x')}(t', x')} - B(t', x')(x-x')} \\
    \lesssim d^{2\al}((t,x), (t', x'))
\end{align}
and $\lesssim$ denotes $\leq C$ where $C$ depends only on $\pra*{\nabla v}_{\al}$, $\lambda$ as well as $\Lambda$, $\al$ and $d$.
\end{thm}
One way to think of \eqref{est1} is as a generalized Taylor expansion. Indeed, rewriting \eqref{est1} slightly as
\begin{align}\label{est2}
    \abs*{(\nabla u (t, x) - \nabla v_{a(t', x')}(t, x)) - (\nabla u(t', x') - \nabla v_{a(t', x')}(t', x')) - B(t', x')(x - x')} \\
    \lesssim d^{2\al}((t,x), (t', x')).
\end{align}
the symmetric matrix $B(t', x')$ plays the role of the Hessian of the function $u(t', \cdot) - v_{a(t', x')}(t', \cdot)$ at the basepoint $(t', x')$. Since $2\al > 1$ it is natural, that an affine correction appears. Moreover, since $2\al > 1$, the matrix $B$ can be seen to be unique and does not depend on $(t, x)$, or more precisely it does not depend on $d\bra*{(t, x), (t', x')}$.
Note that \eqref{est2} is essentially the modelledness condition\footnote{on the level of the solution and not the gradient for $\sigma \equiv 1$} of \cite[Definition 3.1, p.880]{OW19+}, which is an extension of controlled rough paths (cf. \cite{G04}). In \cite{OW19+}, they construct singular products on the level of the stochastic heat equation using probabilistic arguments. They then use the modelledness condition \cite[Definition 3.1, p.880]{OW19+} to lift these singular products to the non-linear setting. Since \eqref{maineq} is not singular, the regularity theory in \cite{OW19} does not rely on such a modelledness. Nevertheless \cref{Thm1} states that such a modelledness holds true. Of course, \eqref{est1} also implies that $\nabla u(t', \cdot) - \nabla v_{a(t', x')}(t', \cdot)$ is differentiable in $x'$ as well as the improved H\"older-regularity in time $\abs*{\nabla u(t, x') - \nabla v_{a(t', x')}(t, x') - \bra*{\nabla u(t', x') - \nabla v_{a(t', x')}(t', x')}} \lesssim \abs*{t-t'}^{\al}$ at $z = \bra*{t', x'}$.

\medskip

From now on, we set $w_{a(t', x')} := u - v_{a(t', x')}$. Then we define the \emph{modelling constant}
\begin{align}
  M :=  \sup_{z = (t', x')} \inf_B \sup_{r > 0} r^{-2\al} \norm*{\nabla w_{a(t', x')} - B_{x'}}_{P_r(z)},
\end{align}
where the infimum ranges over all affine functions $B_{x'}(x) := B(x - x') + b$ with $B \in \R^{d \times d}$ being symmetric and $b \in \R^d$, and $P_r(z)$ is the parabolic cylinder of radius $r$ defined in \cref{sec:not}. In order to prove \cref{Thm1} we show that we have
\begin{align}\label{modconstant}
    M \leq C(d, \lambda, \Lambda, \al, \pra*{\nabla v}_{\al}).
\end{align}
Then it is straightforward to see that the optimal choice is $b = \nabla w_{a(t', x')}(t', x')$ and hence \eqref{modconstant} yields \cref{Thm1}.
Note that the definition of $M$ is essentially the same as the definition of the modelling constant in \cite[Definition 3.1]{OW19+}. Moreover, we want to make the following connection. Apart from the dependence of $w_{a(t', x')}$ on the basepoint $z = (t', x')$, the semi-norm $M$ bears close resemblance to the H\"older-norm defined in \cite[Section 3.3]{K96} where H\"older-regularity is measured in terms of how well a function is approximated by polynomials.

\section{Notation and assumptions}\label{sec:not}
Our notation and assumptions are the same as in \cite[Section 2]{OW19}. For the convenience of the reader we will recall them here.
The non-linearity $A: \R^d \to \R^d$ is assumed to be continuously differentiable and uniformly elliptic in the sense that there exists $\lambda > 0$ such that
\begin{align}\label{ellA}
	\eta \cdot DA(x) \eta \geq \lambda |\eta|^2 \  \text{for all} \ x, \eta \in \mathbb{R}^d
\end{align}
and we have
\begin{align}\label{boundedA}
    |DA(x)\eta| \leq |\eta| \ \text{for all} \ x, \eta \in \mathbb{R}^d
\end{align}
where by $DA$ we denote the Jacobian of $A$. Moreover, we assume that there exists $\Lambda > 0$ such that
\begin{align}\label{Lip}
	|DA(x) - DA(y)| \leq \Lambda |x - y| \ \text{for all} \ x, y \in \R^d.
\end{align}
For $\al \in (0, 1)$ the semi-norm $\pra*{\cdot}_{\al}$ is defined as 
\begin{align} \label{norm}
    [f]_{\al} := \sup_{z \neq z' \in \R \times \R^d} \frac{|f(z) - f(z')|}{d^{\al}(z, z')} < \infty
\end{align}
for all space-time functions $f: \R \times \R^d \to \R$ or vector fields $f: \R \times \R^d \to \R^d$
where
\begin{align}
    d((t, x), (t', x')) := |t-t'|^{\frac{1}{2}} + |x - x'|
\end{align}
denotes the \emph{Carnot--Caratheodory} metric and by abuse of notation $\abs{\cdot}$ refers either to the absolute value or the Euclidean norm on $\R^d$. Naturally, the space $C^{\al}$ denotes all functions $f$ such that
$\pra*{f}_{\al} < \infty$. For $r > 0$ and $z = (t', x')$, by 
\begin{align}
    P_r(z) := (t' - r^2, t') \times B_r(x')
\end{align}
we denote the parabolic cylinder centered around $z$ with radius $r$. Then  $\norm*{\cdot}_{P_r(z)}$ is the supremum norm on $P_r(z)$. We will also frequently write $[f]_{\al, P_r(z)}$ which is the same semi-norm as in \eqref{norm} restricted to $P_r(z)$.

\medskip

For $y \in \mathbb{R}^d$ we define the spatial increment operator $\delta_y$ as
\begin{align}\label{increment}
    \delta_y f(t, x) := f(t, x+y) - f(t, x)
\end{align}
where $f$ is either a scalar or a vector field. Then, by the mean value theorem, we can linearize our non-linearity according to
\begin{align}
	\delta_y A(\nabla u) = a_y \nabla \delta_y u
\end{align}
where
\begin{align}
    a_y(t, x) = \int_0^1 DA(\theta \nabla u(t, x + y) + (1-\theta)\nabla u(t,x)) \dx\theta \in \R^{d \times d}.
\end{align}

\medskip

The assumptions on the Jacobian $DA$ translate to estimates on $a_y$ as follows.
We have
\begin{align}
	\eta \cdot a_y(t,x) \eta \geq \lambda |\eta|^2 \ \text{for all} \ x, \eta \in \mathbb{R}^d, t \in \R
\end{align}
as well as
\begin{align}
    |a_y(t, x)\eta| \leq |\eta| \ \text{for all} \ x, \eta \in \mathbb{R}^d, t \in \R
\end{align}
and
\begin{align}\label{Hoe}
	[a_y]_{\al} \leq [\nabla u]_\al.
\end{align}

\medskip

Let $\psi$ be a smooth, positive and radially symmetric mollifier that satisfies $\supp \psi \subset B_1(0)$ as well as $\int_{\R^d} \psi \dx{x} = 1$. The radial symmetry assumption also ensures that first moments vanish, i.e. for all $i = 1, \dots, d$ it holds
$\int_{\R^d} \psi(x) x_i \dx{x} = 0$. Then we write $\psi_r(x) := \frac{1}{r^d} \psi\bra*{\frac{x}{r}}$
and we define for any function $f$
\begin{align}
    f_r := f_{r, 0} := f * \psi_r
\end{align}
as well as for $i = 1, \dots, d$
\begin{align}
    f_{r, i} := f * \bra*{\partial_i \psi}_r.
\end{align}
For vector fields\footnote{such as a gradient} this notation is to be understood entrywise.
Moreover, we will always write $B_{x'}$ for the affine function
\begin{align}
    B_{x'}(x) := B(x - x') + b
\end{align}
where $B \in \R^{d \times d}$ and $b \in \R^d$. 

\section{The stochastic heat equation}\label{sec:she}
In the work \cite{OW19}, they consider a stationary, spatially periodic Gaussian noise $\xi$ on $\R \times \mathbb{R}^d$, localized in the time interval $\pra*{0, 1}$, that is white in time and colored in space, and such that the corresponding stochastic heat equation
\begin{align}\label{sthe}
    \begin{cases}
        \partial_t v = \Delta v + \xi \\
         v|_{t\leq0} = 0
    \end{cases}
\end{align}
satisfies 
\begin{align}\label{regv}
    \pra*{\nabla v}_{\al} < \infty \ \mathrm{a.s.}
\end{align}
for some $\al \in (\frac{1}{2}, 1)$.
More specifically, they consider a Gaussian field $\xi = \bra*{\xi(f)}_{f}$  with $1-$periodic, positive definite covariance function $K: \R^d \to \R$ such that for smooth test-functions $f, g$
\begin{align}
    \mathbb{E}[\xi(f)\xi(g)] = \int_0^1 \int_{\pra*{0, 1}^d}\int_{\pra*{0, 1}^d} f(t, x) K(x - y) g(t, y) \dx{x} \dx{y} \dx{t}
\end{align}    
and $\xi(f)$ is normally distributed with mean zero.
Such a Gaussian field is easily seen to exist by Kolmogorov's consistency theorem.
The fact that $K$ only depends on one variable yields stationarity of $\xi$. Periodicity of $K$ translates to periodicity of $\xi$. Positive definiteness of $K$ implies that its Fourier transform $\hat{K}$ is real-valued and non-negative and requiring that $K$ is symmetric, i.e. $K(-x) = K(x)$, yields $\hat{K}(-k) = \hat{K}(k)$.
Most importantly, requiring that the Fourier transform of $K$ satisfies 
\begin{align} \label{condK}
    \hat{K}(k) \lesssim (1 + |k|^2)^{-\frac{s}{2}}
\end{align}
for all $k \in (2\pi\mathbb{Z})^d$, where $s = 2\al + d \in (d, d+2)$ finally implies \eqref{regv} (cf. \cite[Lemma 3]{OW19}). 

\medskip

Now we give the proof for \cref{Lem1}.
\begin{proof}[Proof of \cref{Lem1}]
    For any $(t', x') \in \R \times \R^d$ we apply Theorem 1 of \cite{OW19} to $A$ given by
    \begin{align}
        A(\nabla u) = a(t', x')\nabla u
    \end{align}
    where $u = v_{a(t', x')}$\footnote{In this case $A$ is non-deterministic, but since the analysis is pathwise and all the constants depend on $A$ only through the ellipticity constants $\lambda, \Lambda$, which are deterministic, the proofs are unaffected.}. Thus we get a unique solution $v_{a(t', x')}$ that satisfies $\pra*{\nabla v_{a(t', x')}}_{\al} < \infty$ a.s.. Then we can apply Corollary 1 of \cite{OW19} which yields uniformly in $(t', x')$ the estimate
    \begin{align}
        \pra*{\nabla v_{a(t', x')}}_{\al} \leq C(d, \lambda, \Lambda, \al) (\pra*{\nabla v}_{\al} + \pra*{\nabla v}_{\al}^{\frac{\al}{\al_0}})
    \end{align}
    and hence the conclusion.
\end{proof}
\section{Deterministic estimates}
From now on we fix a space-time point $z = (t', x')$. We set $w_{a(t', x')} := u - v_{a(t', x')}$, where we recall that $a(t', x')= DA(\nabla u(t', x'))$. For notational convenience we will drop the subscripts referring to $a(t', x')$. 

\medskip

In the first step, for simplicity, we focus on the spatial part of the modelling constant. The proof is elementary and essentially an extension of the proof of Lemma 1 in \cite{OW19}.
\begin{prop}\label{Prop1}
    Let $x' \in \R^d$. For $f: \mathbb{R}^d \to \R$ such that $\nabla f \in C^{\al}\footnote{only in space in this case}$, we have
    \begin{align}\label{intest}
         \sup_{r > 0} \frac{1}{r^{2\al}}\inf_{B}  \norm*{\nabla f - B_{x'}}_{B_r(x')} &\lesssim \sup_{l > 0} \frac{1}{l^{2\al}} \sup_{\abs*{y} \leq l} \inf_{k \in \R^d} \norm*{\nabla \delta_y f - k}_{B_l(x')} =: N
\end{align}
where $B_{x'}(x) = B(x-x') + b$, $B \in \R^{d \times d}$ symmetric, $ b \in \R^d$.
\end{prop}
\begin{proof}
First of all, we assume that $f$ is smooth.
We denote by $\set*{e_i}_{i = 1, \dots, d}$ the standard orthonormal basis of $\R^d$.
Let $k = k(y, l)$ be the (near) optimal constant for $N$. Fix $i , j = 1, \dots, d$ and $l > 0$. Note that for all $y_1, y_2$ we have the identity $\delta_{y_1 + y_2} \partial_i f = \delta_{y_2}\partial_i f(\cdot + y_1) + \delta_{y_1}\partial_i f$.

Then we estimate
\begin{align}
    \abs*{k_i(2le_j, 2l) - 2 k_i(l e_j, 2l)} &\leq \norm*{k_i(2le_j, 2l) - \delta_{2le_j}\partial_i f}_{B_{2l}(x')} + \norm*{k_i(l e_j, 2l) - \delta_{l e_j} \partial_i f}_{B_{2l}(x')} \\
    &+ \norm*{k_i(l e_j, 2l) - \delta_{l e_j} \partial_i f}_{B_{l}(x')} \\
    &\leq 3(2l)^{2\al} N
\end{align}
and similarly
\begin{align}
    \abs*{k_i(l e_j, 2l) - k_i(l e_j, l)} &\leq \norm*{k_i(l e_j, 2l) - \delta_{l e_j}\partial_i f}_{B_{2l}(x')} + \norm*{k_i(l e_j, l) - \delta_{l e_j}\partial_i f}_{B_{l}(x')} \\
    &\leq 2(2l)^{2\al} N.
\end{align}
Combining these two estimates yields via the triangle inequality
\begin{align}
    \abs*{k_i(2le_j, 2l) - 2 k_i(l e_j, l)} \lesssim N l^{2\al}.
\end{align}
Hence for any $l > 0$ and $n \in \N$ we get 
\begin{align}
    \abs*{\frac{k_i(\frac{l}{2^n} e_j, \frac{l}{2^n})}{\frac{l}{2^n}} - \frac{k_i(\frac{l}{2^{n+1}} e_j, \frac{l}{2^{n+1}})}{\frac{l}{2^{n+1}}}} \lesssim N l^{2\al - 1} \bra*{2^{-n}}^{2\al -1}.
\end{align}
Using our assumption that $2\al - 1 > 0$ we see that the corresponding sequence is Cauchy and thus there exists $a_{ij}(l) \in \R$ such that
\begin{align}\label{limit}
    \frac{k_i(\frac{l}{2^n} e_j, \frac{l}{2^n})}{\frac{l}{2^n}} \to a_{ij}(l) \ \mathrm{for} \ n \to \infty
\end{align}
and, by dyadic summation\footnote{again using the fact that $2\al - 1 > 0$}, this yields
\begin{align}\label{ldep}
    \abs*{\frac{k_i(l e_j, l)}{l} - a_{ij}(l)} \lesssim N l^{2\al - 1}.
\end{align}
Note that $a_{ij}(l)$ is constant on dyadics, i.e.
\begin{align}\label{invdy}
    a_{ij}(l) = a_{ij}(2^{-m}l)
\end{align}
for all $m \in \N$. Feeding the estimate \eqref{ldep} into $N$ we consequently have
\begin{align}
    \norm*{\frac{1}{l} \delta_{le_j} \partial_i f - a_{ij}(l)}_{B_l(x')} \lesssim N l^{2\al - 1}
\end{align}
and, since $f$ is smooth, we infer that
\begin{align}
    a_{ij}(l) \to \partial_{ij}f(x') \ \mathrm{for} \ l \to 0.
\end{align}
Hence we conclude by \eqref{invdy} that $a_{ij}$ is constant and we have $a_{ij} = \partial_{ij}f(x')$. Moreover, the estimate
\begin{align}\label{impest}
    \norm*{\frac{1}{l} \delta_{le_j} \partial_i f - \partial_{ij}f(x')}_{B_l(x')} \lesssim N l^{2\al - 1}
\end{align}
holds.
Now let $x \in B_r(x')$ and we set $y := x - x'$. Then we estimate componentwise using \eqref{impest}
\begin{align}\label{smoothest}
    &\abs*{\partial_i f(x) - \partial_i f(x') -  \nabla \partial_i f(x') \cdot \bra*{x - x'}} \\
    &\leq \sum_{k = 1}^d \abs*{\delta_{y_k e_k} \partial_i f(x'_1, \dots, x'_{k+1} + y_{k+1}, \dots, x'_d + y_d) - \partial_{ki} f(x') y_k} \\
    &\lesssim N \sum_{k=1}^d \abs*{y_k}^{2\al} \lesssim N r^{2\al}.
\end{align}
Now we drop the assumption that $f$ is smooth. To this end, let $\eps > 0$. Let $k$ be the (near) optimal constant for $f$ in $N$. Then we have for all $x \in B_l(x')$ and $\abs*{y} \leq l$ by the triangle inequality
\begin{align}
    \abs*{\delta_y f_{\eps}(x) - k} \leq \norm*{\delta_y f - k}_{B_{2l}(x')} 
\end{align}
and hence
\begin{align}\label{bound}
    \sup_{l > 0} \frac{1}{l^{2\al}} \sup_{\abs*{y} \leq l} \inf_{k^{\eps}} \norm*{\delta_y f_{\eps} - k^{\eps}}_{B_{l}(x')} &\leq \sup_{l > 0} \frac{1}{l^{2\al}} \sup_{\abs*{y} \leq l} \norm*{\delta_y f_{\eps} - k}_{B_{l}(x')} \\
    &\leq \sup_{l > 0} \frac{1}{l^{2\al}} \sup_{\abs*{y} \leq l} \norm*{\delta_y f - k}_{B_{l}(x')}.
\end{align}
Note that, since $\partial_i f$ is H\"older-continuous, $\partial_i f_{\eps}$ converges uniformly to $\partial_i f$ and thus, for fixed $r$, it holds that $\inf_{B^{\eps}} \norm*{\nabla f_{\eps} - B^{\eps}_{x'}}_{B_r(x')}$ converges to $\inf_{B} \norm*{\nabla f - B_{x'}}_{B_r(x')}$. Dividing by $r^{2\al}$ and taking the supremum over $r$ yields lower semicontinuity of the seminorm
\begin{align}\label{lsc}
    \sup_{r > 0} \frac{1}{r^{2\al}} \inf_{B} \norm*{\nabla f - B_{x'}}_{B_r(x')} &\leq \liminf_{\eps \to 0} \sup_{r > 0} \frac{1}{r^{2\al}} \inf_{B^{\eps}} \norm*{\nabla f_{\eps} - B^{\eps}_{x'}}_{B_r(x')}.
\end{align}
We conclude
\begin{align}
    \sup_{r > 0} \frac{1}{r^{2\al}} \inf_{B} \norm*{\nabla f - B_{x'}}_{B_r(x')} &\stackrel{\eqref{lsc}}{\leq} \liminf_{\eps \to 0} \sup_{r > 0} \frac{1}{r^{2\al}} \inf_{B^{\eps}} \norm*{\nabla f_{\eps} - B^{\eps}_{x'}}_{B_r(x')} \\
    &\stackrel{\eqref{smoothest}}{\lesssim} \sup_{l > 0} \frac{1}{l^{2\al}} \sup_{\abs*{y} \leq l} \inf_{k^{\eps}} \norm*{\delta_y f_{\eps} - k^{\eps}}_{B_{l}(x')} \\
    &\stackrel{\eqref{bound}}{\lesssim} \sup_{l > 0} \frac{1}{l^{2\al}} \sup_{\abs*{y} \leq l} \inf_{k} \norm*{\delta_y f - k}_{B_{l}(x')}.
\end{align}
\end{proof}
In order to include time we extend \cref{Prop1} to space-time functions via the following interpolation inequality. 
\begin{cor}\label{cor1}
    For $f : \R \times \R^d \to \R$ such that $\nabla f \in C^{\al}$ we have
    \begin{align}
        \sup_{r > 0} \frac{1}{r^{2\al}} \inf_{B} \norm*{\nabla f - B_{x'}}_{P_r(z)} \lesssim &\sup_{l > 0} \frac{1}{l^{2\al}} \sup_{\abs*{y} \leq l} \inf_{k \in \R^d} \norm*{\nabla \delta_y f - k}_{P_l(z)} \\ 
        &+ \sum_{i = 0}^d \sup_{r > 0} r^{1 - 2\al} \sup_{\abs*{y} \leq r} \norm*{\partial_t \bra*{\delta_y f}_{r, i}}_{P_{r}(z)}
    \end{align}
\end{cor}
\begin{proof}
Let $\bra*{t, x} \in P_r(z)$.
    By \cref{Prop1} there exist a symmetric matrix $B(t', x')$ and a vector $b(t', x')$ such that
    \begin{align}
        \abs*{\nabla f(t', x) - \bra*{B(t', x')\bra*{x - x'} + b(t', x')}} &\lesssim r^{2\al} \sup_{l > 0} \frac{1}{l^{2\al}} \sup_{\abs*{y} \leq l} \inf_{k = k(t)} \norm*{\nabla \delta_y f(t', \cdot) - k}_{B_l(x')} \\
        &\lesssim r^{2\al} \sup_{l > 0} \frac{1}{l^{2\al}} \sup_{\abs*{y} \leq l} \inf_{k \in \R^d} \norm*{\nabla \delta_y f - k}_{P_l(z)}.
    \end{align}
    Via the triangle inequality we split the remainder into a spatial increment and a temporal increment
    \begin{align}
        \abs*{\nabla f(t, x) - \nabla f(t', x)} \leq &\abs*{\nabla f(t, x) - \bra*{\nabla f}_r(t, x)} + \abs*{\bra*{\nabla f}_r(t, x) - \bra*{\nabla f}_r(t', x)} \\ 
        &+ \abs*{\nabla f(t', x) - \bra*{\nabla f}_r(t', x)}.
    \end{align}
    Again by \cref{Prop1} there exist a symmetric matrix $B(t, x')$ and a vector $b(t, x')$ such that
    \begin{align}
        \abs*{\nabla f(t, x) - \bra*{B(t, x')\bra*{x - x'} + b(t, x')}} &\lesssim 
         r^{2\al} \sup_{l > 0} \frac{1}{l^{2\al}} \sup_{\abs*{y} \leq l} \inf_{k \in \R^d} \norm*{\nabla \delta_y f - k}_{P_l(z)}.
    \end{align}
    and, in order to estimate the spatial increment, we appeal to radial symmetry of our mollifier, which implies
    \begin{align}
        B(t, x')(x-x') = \int_{\R^d} \psi_r(x-\zeta) B(t, x')(\zeta - x') \dx{\zeta}
    \end{align}
    and, hence again by \cref{Prop1} and the triangle inequality
    \begin{align}
        &\abs*{\nabla f(t,x) - \bra*{\nabla f}_r(t, x)} \\
        &=\abs*{\nabla f(t,x) - b(t, x') - B(t, x')\bra*{x-x'}  - \int_{\R^d} \psi_{r}(x-\zeta) \bra*{\nabla f(t, \zeta) - b(t, x') - B(t, x')\bra*{\zeta - x'}} \dx{\zeta}} \\
        &\lesssim r^{2\al} \sup_{l > 0} \frac{1}{l^{2\al}} \sup_{\abs*{y} \leq l} \inf_{k \in \R^d} \norm*{\nabla \delta_y f - k}_{P_l(z)}.
    \end{align}
    In order to treat the time difference, we show 
    \begin{align}\label{normeqtime}
        \norm*{\partial_t \bra*{\nabla f}_r}_{P_r(z)} \lesssim \sum_{i = 0}^d \frac{1}{r} \sup_{\abs*{y} \leq r} \norm*{\partial_t \bra*{\delta_y f}_{r, i}}_{P_r(z)}.
    \end{align}
    Let $i = 1, \dots, d$. 
    For any $\hat{t} \in \R$ we compute, appealing to the mean value theorem in space and then Fubini's theorem
    \begin{align}
        \bra*{\partial_i f}_r (\hat{t}, x) - \bra*{\frac{1}{r} \delta_{r e_i} f}_r (\hat{t}, x) &= \int_{\R^d} \psi_r(x - \zeta) \bra*{ \partial_i f(\hat{t}, \zeta) - \int_0^1 \partial_i f(\hat{t}, \zeta + \theta r e_i) \dx{\theta} }\dx{\zeta} \\
        &= - \int_0^1 \int_{\R^d} \psi_r(x -\zeta) \delta_{\theta r e_i} \partial_i f(\hat{t}, \zeta) \dx{\zeta} \dx{\theta}.
    \end{align}
    Integration by parts then yields
    \begin{align}\label{time2}
        \bra*{\partial_i f}_r (\hat{t}, x) - \bra*{\frac{1}{r} \delta_{r e_i} f}_r (\hat{t}, x) = - \int_0^1 \bra*{\frac{1}{r} \delta_{\theta r e_i} f}_{r, i}(\hat{t}, x) \dx{\theta}.
    \end{align}
    Taking the time derivative proves \eqref{normeqtime}.
    
    By the mean value theorem in time, for any $x \in \R^d$ we have
    \begin{align}\label{time1}
        \bra*{\partial_i f}_r(t, x) - \bra*{\partial_i f}_r(t', x) = (t - t') \int_0^1 \partial_t \bra*{\partial_i f}_r(\delta t + (1-\delta)t', x) \dx{\delta}.
    \end{align}
    and combining \eqref{time1} with \eqref{normeqtime} finally yields
    \begin{align}
        \abs*{\bra*{\nabla f}_r(t, x) - \bra*{\nabla f}_r(t', x)} \lesssim r^2 \norm*{\partial_t \bra*{\nabla f}_r}_{P_r(z)} \lesssim \sum_{i=0}^d r \sup_{\abs*{y} \leq r} \norm*{\partial_t \bra*{\delta_y f}_{r, i}}_{P_r(z)}.
    \end{align}
    Hence we have proven that for $\bra*{t, x} \in P_{r}(z)$
    \begin{align}
        \abs*{\nabla f(t, x) - \bra*{B(t', x')\cdot \bra*{x - x'} + b(t', x') }} \lesssim & r^{2\al} \sup_{l > 0} \frac{1}{l^{2\al}} \sup_{\abs*{y} \leq l} \inf_{k \in \R^d} \norm*{\nabla \delta_y f - k}_{P_l(z)} \\ 
        &+ r^{2\al} \sum_{i=0}^d \sup_{l > 0} l^{1 - 2\al} \sup_{\abs*{y} \leq l} \norm*{\partial_t \bra*{\delta_y f}_{l, i}}_{P_{l}(z)}
    \end{align}
\end{proof}
Let $y \in \R^d$. Since $\delta_y u$ satisfies
\begin{align}
    \partial_t \delta_y u - \nabla \cdot a_y \nabla \delta_y u = \partial_t \delta_y v - \nabla \cdot a(t', x') \nabla \delta_y v
\end{align}
we see that $\delta_y w$ satisfies the equation
\begin{align}
	\partial_t \delta_y w - \nabla \cdot a_y \nabla \delta_y w = \nabla \cdot (a_y - a(t', x'))\nabla \delta_y v.
\end{align}
Then, we can write this as a constant coefficient equation 
\begin{align}\label{constcoeffeqw}
	\partial_t \delta_y w - \nabla \cdot a(t', x') \nabla \delta_y w = \nabla \cdot g
\end{align}
by introducing
\begin{align}
    g := (a_y - a(t', x'))\nabla \delta_y u.
\end{align}
Note that \eqref{constcoeffeqw} is now invariant under affine spatial translations, i.e. of functions of the form $\mathrm{aff}(x) := b \cdot x + a$ for some $b \in \R^d, a \in \R$. Hence, we can apply the $C^{1+\al}$ interior Schauder estimate \cite[Theorem 4.8]{L96}, using this invariance, as well as parabolic rescaling, to obtain that, for all $l > 0$ and all space-time points $z = (t', x')$, it holds
\begin{align}
  \quad	\quad \, l^\alpha \inf_{\mathrm{aff}} [\nabla (\delta_y w - \mathrm{aff})]_{\al, P_l(z)} + \inf_{\mathrm{aff}} \norm*{\nabla(\delta_y w - \mathrm{aff})}_{P_l(z)} &\lesssim l^\al \pra*{g}_{\al, P_{2l}(z)} \\
	 &+ l^{-1} \inf_\mathrm{aff} \norm*{\delta_y w - \mathrm{aff}}_{P_{2l}(z)}
\end{align}

which is equivalent to
\begin{align}\label{Nash}
      \quad	\quad \, l^\alpha [\nabla\delta_y w]_{\al, P_l(z)} + \inf_{k \in \R^d} \norm*{\nabla\delta_y w - k}_{P_l(z)} &\lesssim l^\al  \pra*{g}_{\al, P_{2l}(z)} \\
	 &+ l^{-1} \inf_\mathrm{aff} \norm*{\delta_y w - \mathrm{aff}}_{P_{2l}(z)}.
\end{align}
This suggests that we need to further estimate the right hand side which is captured in the next Proposition and which is an extension of \cite[p.73, (34)]{OW19}. The proof is similar.
\begin{prop}\label{Prop2}
Let $z = (t', x')$. For $f: \R \times \mathbb{R}^d \to \R$ continuously differentiable and $y \in \R^d$ we have the estimate
\begin{align}
    \inf_{\mathrm{aff}} \norm*{\delta_y f - \mathrm{aff}}_{P_{l}(z)} \lesssim |y| \inf_B \norm*{\nabla f - B_{x'}}_{P_{2l}(z)}
\end{align}
where $B_{x'}(x) = B(x-x') + b$, $B \in \R^{d \times d}$ is symmetric and $b \in \R^d$.
\end{prop}
\begin{proof}
For $B \in \R^{d \times d}$ symmetric we define
\begin{align}
    \tilde{f}(t, x) := f(t,x) - \bra*{\frac{1}{2} (x-x')\cdot B(x-x') + b \cdot (x-x')}.
\end{align}
Then we compute
\begin{align}
    \nabla \tilde{f}(t, x) = \nabla f(t, x) - (B(x-x') + b)
\end{align}
as well as
\begin{align}
    \delta_y \tilde{f}(t, x) = \delta_y f(t,x) - \bra*{y\cdot B(x-x') + \frac{1}{2} y\cdot By + b \cdot y}.
\end{align}
Notice that for any $y$ the map $x \mapsto y^TB(x-x') + \frac{1}{2} y^{T}By  + b \cdot y$ is again affine. Thus, by the mean value theorem, we have $\norm*{\delta_y \tilde{f}}_{P_{l}(z)} \leq |y| \norm*{\nabla \tilde{f}}_{P_{2l}(z)}$ which yields the statement.
\end{proof}
\subsection{Proof of the main theorem}
We are now able to prove our main theorem.
\begin{proof}[Proof of Theorem \ref{Thm1}]
In the following $B$ always denotes a symmetric matrix.
We write $M = \sup_{z = (t, x)} M_z$ where $M_z = \inf_B \sup_{r > 0} r^{-2\al} \norm*{\nabla w_{a(t', x')} - B_{x'}}_{P_r(z)}$. For now we assume that $M_z$ is finite for any $z$.
By \cref{cor1}, we have
\begin{align}\label{inter}
    \sup_{r > 0} \frac{1}{r^{2\al}}\inf_{B } \norm*{\nabla w - B_{x'}}_{P_r(z)} &\lesssim \sup_{r > 0} \frac{1}{r^{2\al}} \sup_{\abs*{y} \leq r} \inf_{k} \norm*{\nabla \delta_y w - k}_{P_{r}(z)} \\
    &+ \sum_{i = 0}^d \sup_{r > 0} r^{1-2\al} \sup_{\abs*{y} \leq r} \norm*{\partial_t \bra*{\delta_y w}_{r, i}}_{P_r(z)}.
\end{align}
Let $\abs*{y} \leq r \leq l$. 
First, we focus on the second term on the right hand side of \eqref{inter}. To this end, we recall (cf. \eqref{constcoeffeqw}) that $\delta_y w$ satisfies the constant coefficient equation
\begin{align}\label{ccew}
    \partial_t \delta_y w = \nabla \cdot a(t', x') \nabla \delta_y w + \nabla \cdot g
\end{align}
where 
\begin{align}
    g = (a_y - a(t', x'))\nabla \delta_y u
\end{align}
and we estimate for any $i = 0, \dots, d$
\begin{align}\label{timeder}
    \norm*{\partial_t (\delta_{y} w)_{r, i}}_{P_r(z)} &\lesssim  r^{\al-1} \pra*{a(t', x') \nabla \delta_y w  + g}_{\al, P_{2r}(z)} \\
    &\lesssim r^{\al - 1} \bra*{\pra*{\nabla \delta_y w}_{\al, P_{2r}(z)} + \pra*{g}_{P_{2r}(z)}}.
\end{align}
Hence we have
\begin{align}\label{inter2}
     \sup_{r > 0} \frac{1}{r^{2\al}}\inf_{B } \norm*{\nabla w - B_{x'}}_{P_r(z)} \lesssim &\sup_{r > 0} \frac{1}{r^{2\al}} \sup_{\abs*{y} \leq r} \inf_{k} \norm*{\nabla \delta_y w - k}_{P_{r}(z)} \\
     &+ \sup_{r > 0} r^{-\al} \sup_{\abs*{y} \leq r} \bra*{\pra*{\nabla \delta_y w}_{\al, P_{2r}(z)} + \pra*{g}_{P_{2r}(z)}}.
\end{align}
For $\abs*{y} \leq r \leq l$ we estimate by \eqref{Nash}
\begin{align}\label{firststep}
    \inf_{k} \norm*{\nabla \delta_y w - k}_{P_{r}(z)} \leq r^{\al} \pra*{\nabla \delta_y w}_{\al, P_{l}(z)} \lesssim r^{\al} \pra*{g}_{\al, P_{2l}(z)} + \frac{r^{\al}}{l^{1+\al}} \inf_{\mathrm{aff}} \norm*{\delta_y w - \mathrm{aff}}_{P_{2l}(z)}
\end{align}
as well as 
\begin{align}
    \pra*{\nabla \delta_y w}_{\al, P_{l}(z)} \lesssim \pra*{g}_{\al, P_{2l}(z)} + \frac{1}{l^{1+\al}} \inf_{\mathrm{aff}} \norm*{\delta_y w - \mathrm{aff}}_{P_{2l}(z)}.
\end{align}
Now we turn to the estimate of $\pra*{g}_{\al, P_{2l}(z)}$.
We appeal to the Lipschitz continuity of $DA$ (cf. \eqref{Lip}) and recall the definition of $a(t', x')$ as well as $a_y$ in order to estimate
\begin{align}
	|a_y(t, x) - a(t', x')| &\lesssim \int_0^1 |\theta\nabla u(t, x+y) + (1-\theta)\nabla u(t,x) - \nabla u(t', x')| \dx \theta \\
	&\lesssim \int_0^1 \theta |\nabla u(t, x+y) - \nabla u(t, x)| + |\nabla u(t, x) - \nabla u(t', x')| \dx \theta \\
	&\lesssim |y|^{\al} + d((t, x), (t', x'))^{\al}
\end{align}
and thus, for $|y| \leq r \leq l$ we get
\begin{align}
	\norm*{a_y - a(t', x')}_{P_{2l}(z)} \lesssim l^{\al}.
\end{align}
By \eqref{Hoe}, we have
\begin{align}
	[a_y - a(t', x')]_{\al, P_{2l}(z)} = [a_y]_{\al, P_{2l}(z)} \lesssim [\nabla u]_{\al, P_{3l}(z)}.
\end{align}
Moreover, it holds that
\begin{align}
	\norm*{\nabla \delta_y u}_{P_{2l}(z)} \lesssim l^{\al} [\nabla u]_{\al, P_{3l}(z)} 
\end{align}
as well as
\begin{align}
	[\nabla \delta_y u]_{\al, P_{2l}(z)} \lesssim [\nabla u]_{\al, P_{3l}(z)}.
\end{align}
Summing up, we arrive at
\begin{align}\label{estv}
	\pra*{g}_{\al, P_{2l}(z)} = [(a_y - a(t', x'))\nabla \delta_y u]_{\al, P_{2l}(z)} &\leq [a_y - a(t', x')]_{\al, P_{3l}(z)} \norm*{\nabla \delta_y u}_{P_{3l}(z)} \\ 
	&+ \norm*{a_y - a(t', x')}_{P_{3l}(z)} [\nabla \delta_y u]_{\al, P_{3l}(z)} \\ 
	&\lesssim l^{\al} [\nabla u]_{\al, P_{3l}(z)}
\end{align}
for $|y| \leq r \leq l$. 
Combining \eqref{estv} with \eqref{firststep} yields for $|y| \leq r \leq l$
\begin{align}\label{secondstep}
    r^{\al}\pra*{\nabla \delta_y w}_{\al, P_{l}(z)}
    &\lesssim  l^{2\al} \pra*{\nabla u}_{\al, P_{3l}(z)} + \frac{r^{\al}}{l^{1+\al}} \inf_{\mathrm{aff}} \norm*{\delta_y w - \mathrm{aff}}_{P_{2l}(z)}. 
\end{align}
Then using \cref{Prop2} we further estimate \eqref{secondstep} to the effect that
\begin{align}\label{thirdstep}
    r^{\al}\pra*{\nabla \delta_y w}_{\al, P_{l}(z)}&\lesssim l^{2\al} \pra*{\nabla u}_{\al, P_{3l}(z)} + \bra*{\frac{r}{l}}^{1+\al} \inf_B \norm*{\nabla w - B_{x'}}_{P_{4l}(z)}.
\end{align}
Summing up, we get by \eqref{timeder} and \eqref{thirdstep}
\begin{align}\label{fourthstep}
    &\frac{1}{r^{2\al}} \sup_{\abs*{y} \leq r} \inf_k \norm*{\nabla \delta_y w - k}_{P_r(z)} + \sum_{i = 0}^d r^{1-2\al} \sup_{\abs*{y} \leq r} \norm*{\partial_t \bra*{\delta_y w}_r}_{P_r(z)} \\
    &\lesssim \bra*{\frac{l}{r}}^{2\al} \pra*{\nabla u}_{\al, P_{3l}(z)} + \bra*{\frac{r}{l}}^{1-\al} \frac{1}{l^{2\al}} \inf_B \norm*{\nabla w - B_{x'}}_{P_{4l}(z)}.
\end{align}
Let $K \geq 1$ and set $l = Kr$. By introducing $M'_z := \sup_{r > 0} \frac{1}{r^{2\al}} \inf_B \norm*{\nabla w - B_{x'}}_{P_{r}(z)}$
we estimate combining \eqref{inter} and \eqref{fourthstep}
\begin{align}
	M'_z \lesssim K^{2\al} \pra*{\nabla u}_{\al}  + K^{\al - 1}  M_z
\end{align}
Thus, after using \eqref{aprioriest} as well as \cref{Lem1} we arrive at
\begin{align}
	M'_z \lesssim K^{2\al} C\bra*{\pra*{\nabla v}_{\al}} + K^{\al - 1} M_z
\end{align}
where the implicit constant in $\lesssim$ depends only on $d, \lambda, \Lambda$ and $ \al$ and $C\bra*{\pra*{\nabla v}_{\al}}$ depends polynomially on $\pra*{\nabla v}_{\al}$. 
Since $\al > \frac{1}{2}$, using the reasoning in Step 3 in the proof of Lemma 3.6 of \cite{OW19+} the matrix $B$ is independent of $r$, i.e. 
\begin{align}\label{final1}
    M_z \lesssim  M'_z
\end{align}
and hence
\begin{align}
    M_z \lesssim K^{2\al} C\bra*{\pra*{\nabla v}_{\al}} + K^{\al - 1} M_z.
\end{align}
Choosing $K$ sufficiently large, and since $\al < 1$, as well as the assumption that $M_z$ is finite, we can absorb $K^{\al - 1} M_z$ into the left hand side and end up with
\begin{align}\label{apriori}
	M_z \lesssim C(\pra*{\nabla v}_{\al})
\end{align}
and taking the supremum in $z$ we get
\begin{align}
    	M \lesssim C(\pra*{\nabla v}_{\al}).
\end{align}
In order to get rid of the assumption that $M_z$ is finite we appeal to the following approximation argument. Up to now we made the abbreviation $v = v_{a(t', x')}$ but from now on by $v$ we mean the solution to the SHE \eqref{sthe}. As in the proof of \cite[Theorem 1]{OW19}, we can consider smooth approximations $u_{\eps}$, $v_{a(t', x'), \eps}$ and $v_{\eps}$ to $u$, $v_{a(t', x')}$ and $v$ such that
\begin{align}
    \nabla \bra*{u_{\eps} - v_{a(t', x'), \eps}} \to \nabla \bra*{u - v_{a(t', x')}} = \nabla w
\end{align}
uniformly on compact sets and such that $\pra*{\nabla v_{\eps}}_{\al} \lesssim \pra*{\nabla v}_{\al}$. Then the corresponding constant $M_{z, \eps}$ is finite for all $z$ and hence \eqref{apriori} applies and we have
\begin{align}
    M_{\eps} \lesssim C(\pra*{\nabla v_{\eps}}_{\al}) \lesssim C(\pra*{\nabla v}_{\al}).
\end{align}
Hence, by the same reasoning as in \cref{Prop1}, we conclude by lower-semicontinuity
\begin{align}\label{final2}
    M \lesssim \liminf_{\eps \to 0} M_{\eps} \lesssim C(\pra*{\nabla v}_{\al}).
\end{align}
In other words, \eqref{final2} means that for all $r > 0$ and all space-time points $z = (t', x')$ there exists a family of symmetric matrices $\bra*{B(t', x')}_{(t', x')}$ and a family of vectors $\bra*{b(t', x')}_{(t', x')}$, such that we have
\begin{align}
    |\nabla u(t, x) - \nabla v_{a(t', x')}(t, x) - (B(t', x')(x - x') + b(t', x'))| \lesssim r^{2\al} \ \text{for} \ (t, x) \in P_r(z).
\end{align}
Letting $r \to 0$ implies $(t, x) \to (t', x')$ and hence the optimal $b$ is of the form
\begin{align}
    b(t', x') = \nabla u(t', x') - \nabla v_{a(t', x')}(t', x').    
\end{align}
Then we can choose $r = d((t, x), (t', x'))$ and thus 
\begin{align}
    |\nabla u(t, x) - \nabla v_{a(t', x')}(t, x) - (\nabla u(t', x') - \nabla v_{a(t', x')}(t', x')) - B(t', x')(x - x')| \\
    \lesssim d^{2\al}((t, x), (t', x')),
\end{align}
which we wanted to prove.
\end{proof}
\section*{Acknowledgments}
The author would like to thank Felix Otto for helpful discussions and for bringing this problem to his attention. 
\bibliographystyle{amsplain}
\bibliography{refs}

\end{document}